\documentclass[11pt]{article}
\usepackage[utf8]{inputenc}
\usepackage{mathtools,amsmath,color,amssymb,amsthm, babel, amsfonts, amssymb}
\usepackage{todonotes}
\usepackage{enumitem}
\usepackage{latexsym, ifthen}
\usepackage{graphicx}
\usepackage[ocgcolorlinks,colorlinks=true,urlcolor=blue,breaklinks=true]{hyperref}
\usepackage[top=1in,bottom=1in,left=1in,right=1in]{geometry}

\let\Horig\H

\def\reals{{\mathbb{R}}}

\def\H{{\cal H}}

\newtheorem{theorem}{Theorem}[section]

\newtheorem{lemma}[theorem]{Lemma}

\def\reals{{\mathbb R}}

\begin{document}
%--------------------------------------------------------
\title{Covering points by hyperplanes and related problems\thanks{A conference version of this paper appeared in the Proceedings of the 38th Symposium on Computational Geometry (SoCG 2022): {\bf 57}, 1--7.}
\thanks{%
Work by Z.P. partially supported by GA\v {C}R grant 22-19073S and by Charles University projects PRIMUS/21/SCI/014 and UNCE/SCI/022.  Work by M.S. partially supported by ISF grant~260/18.}
  }
  
\author{Zuzana Pat\'akov\'a\thanks{%
  Department of Algebra, Faculty of Mathematics and Physics, Charles University, Sokolovská~83, 186 75 Praha, Czech Republic; E-mail: {\sf patakova@karlin.mff.cuni.cz}.}\and%
   Micha Sharir\thanks{%
  School of Computer Science, Tel Aviv University, Tel~Aviv 69978, Israel; E-mail: {\sf michas@tau.ac.il}.}
}

\date{}
\maketitle
%--------------------------------------------------------
\begin{abstract}
For a set $P$ of $n$ points in~$\reals^d$, for any $d\ge 2$, a hyperplane $h$ is called 
\emph{$k$-rich} with respect to $P$ if it contains at least $k$ points of $P$. Answering 
and generalizing a question asked by Peyman Afshani, we show that if the number of $k$-rich 
hyperplanes in $\reals^d$, $d \geq 3$, is at least $\Omega(n^d/k^\alpha + n/k)$, with a sufficiently large constant of 
proportionality and with $d\le \alpha < 2d-1$, then there exists a $(d-2)$-flat that contains 
$\Omega(k^{(2d-1-\alpha)/(d-1)})$ points of $P$. We also present upper bound constructions
that give instances in which the above lower bound is tight. 
An extension of our analysis yields similar lower bounds for $k$-rich spheres or $k$-rich flats. 
\end{abstract}

%--------------------------------------------------------
\section{Introduction} 

Let $P$ be a set of $n$ points in~$\reals^d$. A hyperplane $h$ is called \emph{$k$-rich} 
with respect to $P$ if it contains at least $k$ points of $P$. 
Assume that the number of $k$-rich hyperplanes is at least $\Omega(n^d/k^{d+1} + n/k)$,
with a sufficiently large constant of proportionality. Is there a lower-dimensional flat 
containing ``a lot of points'' of $P$? This question was raised by Peyman Afshani (personal 
communication), motivated by his recent work \cite{ABDN} on point covering problems.
We answer Afshani's problem in the affirmative, in the following stronger form.
%------------------------------------------
\begin{theorem} \label{thm:main}
Let $d \geq 3, k \geq d$ be integers, and $d\leq\alpha < 2d-1$.
Let $P$ be a set of $n$ points in~$\reals^d$, for which the number of 
$k$-rich hyperplanes is at least $c(n^d/k^{\alpha} + n/k)$, for some sufficiently 
large constant $c$ (depending only on $d$). Then there exists a $(d-2)$-flat 
that contains $\Omega\left(k^{(2d-1-\alpha)/(d-1)}\right)$ points of $P$.
\end{theorem} 
%-------------------------------------------
For $\alpha = d+1$, as in Afshani's original question, the lower bound is $\Omega(k^{(d-2)/(d-1)})$.
We also present two upper bound constructions that give instances of the problem
in which the bound in Theorem~\ref{thm:main} is tight. The first instance involves
$\alpha = d+1$ (as in Afshani's original question) and certain values of $k$,
and in the second instance we have $\alpha = d = 3$.

We also extend our analysis to the case of $k$-rich spheres  or $k$-rich flats (a sphere/flat is $k$-rich if it contains at least $k$ points of $P$). In particular, we show (see Theorem~\ref{thm:spheres}) that if the number of
$k$-rich $(d-1)$-spheres is at least $c(n^{d+1}/k^{\alpha} + n/k)$, for $d+1\leq\alpha < 2d+1$ 
and for some sufficiently large constant $c$, then there exists a $(d-2)$-sphere 
that contains $\Omega\left(k^{(2d+1-\alpha)/d}\right)$ points of $P$.  Similarly we show (see Theorem \ref{t:flats}) that if the number of $k$-rich $(t-1)$-flats is at least $c(n^t/k^{\alpha} + n/k)$, for $d \geq t\geq 3, k \geq t, t \leq \alpha < 2t-1$, and for some sufficiently large constant $c$, then there exists a $(t-2)$-flat containing $\Omega\left(k^{(2t-1-\alpha)/(t-1)}\right)$ points of $P$.

The result is interesting by itself, but it may also  find a potential application in the so called \emph{hyperplane cover} problem, one of the classical problems in computational geometry: given a set $S$ of $n$ points in~$\reals^d$ and a number $h$, can we find $h$ hyperplanes that cover all points of $S$? It is a geometric variant of a set cover problem, and it was shown that already for $d=2$ the hyperplane cover problem is both NP-hard \cite{MT} and APX-hard \cite{KAR}. 
However, several FPT-algorithms (in the fixed parameter $h$) are known, best of which is \cite{WLCh}. In the special cases $d=2$ and $d=3$ it has been further improved \cite{ABDN}.
The improvement is based on incidence bounds and builds on a simple observation that given a hyperplane cover of cardinality $h$, some of the hyperplanes might contain many more points than the others. The main idea is to deal with such hyperplanes first and the performance of the algorithm depends on the number of such hyperplanes. 
 For example, it follows from the Szemerédi-Trotter theorem that there are at most $O(n^d/k^3)$ $k$-rich hyperplanes defined by $n$ points in~$\reals^d$.  However, in the approach of \cite{ABDN} this bound  turned out to be useful only in the plane, in which case the exponent of~$n$ is strictly smaller than the exponent of~$k$. The 3-dimensional case is  treated  using another incidence bound \cite{ET}, but this approach also does not extend to higher dimensions \cite{ABDN}. What could help for $d \geq 4$ is to show that if there are too many rich hyperplanes the points cannot be distributed arbitrarily, in fact,  many of them must lie on a common lower dimensional flat.
The results of our paper address this issue.

The problem is also closely related to the problem of bounding the number of incidences 
between $n$ points and $m$ hyperplanes, and we will indeed use tools from incidence theory 
to tackle this problem. A major hurdle in obtaining sharp point-hyperplane incidence bounds, 
in $d\ge 3$ dimensions, is the possibility that there exists a $(d-2)$-flat that contains 
many of the points and is contained in many of the hyperplanes. In the worst case all the 
$n$ points could be contained in such a flat, and all the $m$ hyperplanes could contain 
the flat, and then the number of incidences would be $nm$, the largest possible value. 
To obtain sharper bounds one usually needs to require that no $(d-2)$-flat contains 
too many points, or is not contained in too many hyperplanes, or to impose other 
restrictions on the setup. See \cite{AA,ApS,BK,EGS,ET,Rud,SiSu} 
for a sample of earlier works on this topic. For example, better bounds can be obtained
if the points are restricted to be vertices of the arrangement of the hyperplanes~\cite{AA},
or when the incidence graph between the points and hyperplanes does not contain a
complete bipartite subgraph of some small size (see~\cite{BK}). Improved bounds can also
be obtained by assuming that no lower-dimensional flat is contained in too many hyperplanes, 
or does not contain too many points~\cite{EGS}.
Some of these works also derive lower bounds,
but for different quantities, which do not seem directly related to the setup
considered in this paper. See, for example, Apfelbaum and Sharir~\cite{ApS} and
Brass and Knauer~\cite{BK}
for lower bounds on the maximum size of a complete bipartite
subgraph in the incidence graph of points and hyperplanes.

%--------------------------------------------------------
\section{Proof of Theorem \ref{thm:main}} \label{sec:proof}

Let $P$ be a set of $n$ points in~$\reals^d$ that has many
$k$-rich hyperplanes, in the sense of Theorem~\ref{thm:main}, and let $\ell$ denote the 
maximum number of points of $P$ contained in any $(d-2)$-flat. We seek a lower bound on $\ell$.

\paragraph*{Overview of the proof} Before we dive into the details, we describe the overall idea first. Let $H$ be the set of all $k$-rich hyperplanes spanned by $P$. By a simple argument we show that $H$ is finite and then we establish a lower and an upper bound on the number of incidences between $P$ and $H$. Comparing these bounds yields the desired result. As the lower bound on the number of incidences is trivially $k|H|$, the actual work here is to obtain a reasonable upper bound---for that we use simplicial partitions (Theorem \ref{t:cuttings}), point-hyperplane duality, and the Cauchy-Schwartz inequality.

We start with a simple incidence bound.

%------------------------------------
\begin{lemma}\label{l:simple_bound}
Let $P$ and $H$ be finite sets of points and hyperplanes in $\reals^d$, respec-tively, such that no $(d-2)$-flat contains more than $\ell$ points of $P$.
We have the following simple bound on the number $I(P,H)$ of incidences
between the points of $P$ and the hyperplanes of $H$.
\begin{equation}
  I(P,H) = O\left(|H||P|^{1/2}\ell^{1/2} + |P|\right). \label{eq:simple_bound_l} 
\end{equation}
\end{lemma}
%------------------------------------

\begin{proof}
This is a simple geometric application of the well known K\Horig ov\'ari-S\'os-Tur\'an 
Theorem (see, e.g.,~\cite{PA,KST}), which says that a $K_{t,2}$-free bipartite graph 
with $n$ left and $m$ right vertices has at most $O(mn^{1/2}t^{1/2} + n)$ edges. 
The proof is based on the observation that the incidence graph between $P$ and $H$ 
does not contain $K_{\ell+1,2}$ as a subgraph. Indeed, any pair of non-parallel hyperplanes
from $H$ intersect in a $(d-2)$-flat, which, by assumption, contains at most $\ell$ points of $P$.
\end{proof}

%---------------------------------------
\paragraph*{Using simplicial partitions}

We now proceed to sharpen the upper bound in Lemma~\ref{l:simple_bound}.
We recall the following result, due to Matou\v{s}ek~\cite{Ma:ept}.
%--------------------------------------------------------
\begin{theorem}\label{t:cuttings}
Let $Q$ be a set of $m$ points in~$\reals^d$, for any $d\ge 2$, and let $1 < r \leq m$ be a given
parameter. Then $Q$ can be partitioned into $q \leq 2r$ subsets, $Q_1, \ldots, Q_q$, 
so that, for each $i$, $m/(2r) \leq |Q_i| \leq m/r$, and $Q_i$ is contained in the relative 
interior of a (possibly lower-dimensional) simplex $\Delta_i$, so that every hyperplane 
crosses (i.e., intersects but does not contain) at most $O(r^{1-1/d})$ of these simplices. 
\end{theorem}
%--------------------------------------------------------
The partition in Theorem~\ref{t:cuttings} is referred to as a \emph{simplicial partition} of $Q$.
We remark that the theorem guarantees that none of the simplices is a single point when $r \leq m/4$. 
This result has more recently been refined by Chan~\cite{Chan}, but the original version
suffices for our purpose.

%--------------------------------------------------------

\begin{proof}[Proof of Theorem \ref{thm:main}]
%---------------------------------------------------------
First note that if there is a $(d-2)$-flat containing at least $k$ points of $P$, 
the theorem trivially holds, as we then have $\ell \geq k \geq k^{(2d-1-\alpha)/(d-1)}$,
since $\alpha \geq d$. Hence we can assume that each $(d-2)$-flat contains at 
most $k-1$ points of $P$. This guarantees that the number of all $k$-rich hyperplanes 
(with respect to $P$) is finite, as every $k$-tuple of points of $P$ spans at most one 
$k$-rich hyperplane. 

Let then $H$ be the finite set of all $k$-rich hyperplanes, $k \geq d$, set $m := |H|$, 
and recall that we assume that $m = |H| \ge c\left( n^d/k^\alpha + n/k \right)$, for some 
sufficiently large constant $c$ (that depends on $d$) and for some $d\le \alpha < 2d-1$.

Our strategy is to derive an upper bound on the number of incidences between the points
of $P$ and the hyperplanes of $H$, and combine it with the obvious lower bound $mk$ on
this number, which follows since each of these hyperplanes is $k$-rich. The combination
of these bounds will lead to the desired lower bound on $\ell$.

We apply standard geometric duality in $\reals^d$ and get a set $H^*$ of $m$ dual points 
and a set $P^*$ of $n$ dual hyperplanes. The dual version of the fact that no $(d-2)$-flat contains 
more than $\ell$ points of $P$ is that no line is contained in more than $\ell$ hyperplanes
of $P^*$. We also know, as just mentioned, that $I(P,H) \geq mk$, as each primal hyperplane in $H$ contains
at least $k$ points of~$P$.

We fix some $r$, which we determine later, and apply Theorem \ref{t:cuttings} in the
dual setting. We obtain $q\le 2r$ subsets $H_1^*,\ldots,H_q^*$, so that $m/(2r)\le |H_i^*|\le m/r$ 
for each $i=1,\ldots,q$, and each hyperplane crosses $O(r^{1-1/d})$ of the corresponding
simplices. Denote also by $P_i^*$ the set of dual hyperplanes that cross the $i$-th simplex
$\Delta_i \supset H_i^*$, for each $i$. Let $P_i$ and $H_i$ denote the corresponding sets 
of points and hyperplanes in the primal space.

The number of incidences of dual points inside the partition cells and dual hyperplanes 
crossing the corresponding simplices can be bounded as follows:
%---------------------------------------------
\begin{align} \label{eq:interior_incid:ddim}
\sum_{i=1}^q I(H_i^*,P_i^*) & = \sum_{i=1}^q I(P_i,H_i) = 
O\left(\sum_{i=1}^q |P_i|^{1/2}|H_i|\ell^{1/2} + \sum_{i=1}^q |P_i| \right) \\
& = O\left( m(\ell n)^{1/2}r^{-1/(2d)} + nr^{1-1/d} \right) \nonumber .
\end{align}
%---------------------------------------------
The first inequality follows by applying the bound (\ref{eq:simple_bound_l}) of
Lemma \ref{l:simple_bound} in the primal.
For the second inequality we use the property that each dual hyperplane crosses at most 
$O(r^{1-1/d})$ cells, so we have, using the Cauchy-Schwarz inequality, and recalling that $q\le 2r$,
\begin{eqnarray*}
|H_i| = |H_i^*| \leq \frac{m}{r}, \qquad \sum |P_i| = \sum |P_i^*|= O(r^{1-1/d}n), \qquad\text{and} \\
\sum_{i=1}^q |P_i|^{1/2} \le \left( \sum_{i=1}^q |P_i| \right)^{1/2}(2r)^{1/2} = 
O(n^{1/2}r^{(2d-1)/(2d)}), \label{eq:inside_incid}
\end{eqnarray*}
and the second inequality follows.

It remains to count the incidences between points in a cell (simplex) and hyper-planes that contain the simplex. Any such simplex $\sigma$ is $j$-dimensional, for some $1\le j \le d-1$ 
(zero-dimensional simplices do not arise when $r \leq m/4$).
When $j=d-1$, each such $\sigma$ is contained in at most one hyperplane of $P^*$, 
contributing in total at most $m'$ incidences, where $m'$ is the number of dual points
contained in such cells. When $1\le j\le d-2$, $\sigma$ spans (affinely) a $j$-flat $g$,
which cannot be contained in more than $\ell$ dual hyperplanes in $P^*$, for otherwise 
any line in $g$ would also be contained in these hyperplanes, contrary to our assumption.
Hence the number of resulting incidences is at most $\ell m''$, where $m''$ is the number 
of dual points contained in such simplices. In total, all the lower-dimensional simplices 
contribute at most $\ell m$ incidences.
 
Hence, combining this with (\ref{eq:interior_incid:ddim}), we get:
\begin{equation}
mk \leq I(P,H) \leq O\left( m\ell^{1/2} n^{1/2}r^{-1/(2d)} + r^{1-1/d}n \right) + \ell m . \label{eq:incid:ddim}
\end{equation}
We now balance the first two terms by choosing
\[
r:= \left( \frac{\ell m^2}{n}\right)^{d/(2d-1)}.
\]
For this to make sense $r$ has to be between $1$ and $m/4$. We note that $r < 1$
when $m < (n/\ell)^{1/2}$ and $r > m/4$ when $m > c_1n^d/\ell^d$, for some constant $c_1$
that depends on $d$. In the former case we take $r=1$ and the first two terms become $O(n)$. 
(Note that the choice $r=1$ corresponds to a direct application of Lemma~\ref{l:simple_bound}.)
In the latter case we take $r=m/4$ and the first two terms become 
\[
O(m^{(2d-1)/(2d)}\ell^{1/2} n^{1/2} + m^{1-1/d}n) = O\left( m^{(2d-1)/(2d)}\ell^{1/2} n^{1/2} \right) = O(m\ell),
\]
where both inequalities hold because $m > c_1n^d/\ell^d$.
When neither of these two extreme cases occurs, 
the first two terms become $O(m^{(2d-2)/(2d-1)}\ell^{(d-1)/(2d-1)}n^{d/(2d-1)})$. Altogether we thus get
\begin{equation}
mk \leq O\left( m^{(2d-2)/(2d-1)}\ell^{(d-1)/(2d-1)}n^{d/(2d-1)} + m\ell + n \right) . \label{eq:incidx:ddim}
\end{equation}
The inequality in (\ref{eq:incidx:ddim}) implies that either 
$\ell = \Omega(k) = \Omega\left(k^{(2d-1-\alpha)/(d-1)}\right)$, since $\alpha \ge d$, or
\begin{equation} \label{eq:mup:ddim}
m = O\left( \frac{\ell^{d-1} n^d}{k^{2d-1}} + \frac{n}{k} \right),
\end{equation}
where we have distinguished two cases depending on whether the first or the last 
term in the right-hand side of  (\ref{eq:incidx:ddim}) dominates.
Let $c'$ be the O-notation constant from (\ref{eq:mup:ddim}).
Since we assume that $m \ge c(n^d/k^\alpha + n/k)$, where $c$ is a sufficicently large constant, we get

\[
 c\left (\frac{n^d}{k^\alpha} + \frac{n}{k}\right) \le c'\left( \frac{\ell^{d-1} n^d}{k^{2d-1}} + \frac{n}{k} \right).
\]
For $c \geq c'$ it simplifies to
\[
  \frac{cn^d}{k^\alpha} \le \frac{cn^d}{k^\alpha} + (c-c')\frac{n}{k} \le c'\frac{\ell^{d-1} n^d}{k^{2d-1}},
\]
which implies that $\ell = \Omega(k^{(2d-1-\alpha)/(d-1)})$.
This completes the proof of Theorem~\ref{thm:main}.
\end{proof}

%----------------------------------------------------------------------------------------------------------
\subsection{Upper bound constructions} \label{sec:up}

\paragraph*{First construction}
The following construction only handles the case $\alpha = d+1$ (the original question 
of Afshani) and certain restricted values of $k$; it is a variation of a construction of Elekes~\cite{Ele}.
  
Fix two integer parameters $u > v\ge 1$ where $v$ is a suitable constant. 
Let $P$ be the set of vertices of the $u\times \cdots \times u\times duv$ integer grid in $\reals^d$. That is,
\[
P = \{(i_1,\ldots, i_d) \mid 0\leq i_1, \ldots, i_{d-1} \leq u-1,\;0\leq i_d \leq duv-1\}.
\]
We have $n := |P| = du^dv$ and we set $k := u^{d-1}$. Any hyperplane of the form 
$x_d = a_1x_1 + a_2x_2 + \cdots + a_{d-1}x_{d-1} + a_d$, with integer coefficients 
satisfying $0\le a_i \le v-1$, for $1\le i \le d-1$, and $0\le a_d \le uv-1$, is trivially seen 
to be $k$-rich with respect to $P$. 
Hence the number of $k$-rich hyperplanes is at least $uv^d$. On the other hand, we have
%----------------------- 
\begin{equation*}
\frac{n^d}{k^{d+1}} = \frac{d^du^{d^2}v^d}{u^{(d+1)(d-1)}} = d^duv^d. \label{eq:rich_hyperplanes}
\end{equation*}
%----------------------- 
It is easily verified that a $(d-2)$-flat $\lambda$ that is not vertical (i.e., not parallel to the 
$x_d$-axis) contains at most $u^{d-2}$ points of $P$, and that a vertical $(d-2)$-flat can contain 
$u^{d-3}duv = O(u^{d-2}) = O(k^{(d-2)/(d-1)})$ points of $P$ (but not more).
Hence, setting $\ell$ to be $ck^{(d-2)/(d-1)}$, for a suitable coefficient $c$,
we have a construction with at least $\frac{n^d}{d^dk^{d+1}}$ $k$-rich hyperplanes,
but no $(d-2)$-flat contains more than $ck^{(d-2)/(d-1)}$ points of $P$. In other words, 
our bound is asymptotically worst-case tight for this special setup.

We remark that in this construction we have $k = \Theta(n^{1-1/d})$, 
so one still needs to show that the bound is tight for other values of $k$. 
We leave this as an open problem.

\paragraph*{Second construction}
 
A more significant open challenge is to extend the construc-tion to other values
of $\alpha$ in the range $d \leq \alpha < 2d-1$.
We make a first step towards this goal, by presenting, for $\alpha = d = 3$,
another simple construction. Let $k \geq 3$, $k \geq u \geq 2$ be integer parameters.
Consider a set $L$ of $u$ pairwise skew lines in $\reals^3$, each containing 
$k$ distinguished points. Let $P$ be the set of all these points. We have 
$n:= |P| = ku$. Note that there are infinitely many $k$-rich planes with respect 
to $P$ as any plane containing a single line from $L$ is $k$-rich. On the other hand, 
it follows from the construction that no line contains strictly more than $k$ points 
of $P$. Indeed, any line not contained in $L$ intersects at most $k$ lines from $L$ 
(since $u \leq k$), so it can contain at most $k$ points of $P$. Hence, $\ell = k$,
which shows that the bound in Theorem~\ref{thm:main} is tight for $d= \alpha = 3$ 
and $n/2 \geq k \geq n^{1/2}$.

%----------------------------------------------------------------------------------------------------------
\section{The case of spheres}

The analysis can be extended to the case of spheres in a straightforward manner. 
Specifically, we have a set $P$ of
$n$ points in~$\reals^d$, for $d\ge 3$. We say that a sphere $\sigma$ is $k$-rich with
respect to $P$ if it contains at least $k$ points of $P$. The goal now is to show that if
there are many $k$-rich $(d-1)$-spheres in $\reals^d$ then there exists a $(d-2)$-sphere that contains many points of~$P$.
The concrete statement is: 
%----------------------------
\begin{theorem} \label{thm:spheres}
Let $d \ge 3, k \ge d+1$ be integers, and let $d+1\le \alpha < 2d+1$ be a parameter. 
Let $P$ be a set of $n$ points in~$\reals^d$, for which the number of 
$k$-rich $(d-1)$-spheres is at least $c(n^{d+1}/k^{\alpha} + n/k)$, for some sufficiently 
large constant $c$ that depends on $d$. Then there exists a $(d-2)$-sphere 
that contains $\Omega\left(k^{(2d+1-\alpha)/d}\right)$ points of $P$.
\end{theorem}
%----------------------------

Note that if there is a $k$-rich $(d-2)$-sphere, the theorem holds trivially,
as we then have $\ell \geq k$ and $\alpha \geq d+1$. Hence we can assume that no $(d-2)$-sphere 
is $k$-rich, which implies, as in the case of hyperplanes, that the number of $k$-rich $(d-1)$-spheres is finite.

The proof is an adaptation of the preceding analysis.
Let $P$ be a set of $n$ points in~$\reals^d$, for $d\ge 3$, that has many $k$-rich 
$(d-1)$-spheres, in the sense of Theorem~\ref{thm:spheres}. Let $\ell$ denote the maximum 
number of points of $P$ contained in any $(d-2)$-sphere. As before, we seek a lower bound on~$\ell$.

Lemma~\ref{l:simple_bound} continues to 
hold in the case of spheres, with more or less the same proof, using the obvious property that
two non-disjoint $(d-1)$-spheres intersect in a $(d-2)$-sphere or a single point. 
To sharpen the bound we proceed as follows.

Let $\Sigma$ be the set of all $k$-rich $(d-1)$-spheres, $k \geq d+1$, and recall that we assume that
$m := |\Sigma| \ge c\left( n^{d+1}/k^\alpha + n/k \right)$, 
for some sufficiently large constant $c$ (that depends on $d$) and for $d+1\le \alpha < 2d+1$.

We apply the standard lifting transform $(x_1,\ldots, x_d) \mapsto (x_1,\ldots, x_d, x_1^2 + \cdots + x_d^2)$,
which transforms $(d-1)$-spheres in $\reals^d$ to hyperplanes in $\reals^{d+1}$. 
Applying standard duality in $\reals^{d+1}$, we get
a set $\Sigma^*$ of $m$ dual points and a set $P^*$ of $n$ dual hyperplanes in $\reals^{d+1}$.
The lifted-dual version of the fact that no $(d-2)$-sphere, which is lifted to a $(d-1)$-flat in
$\reals^{d+1}$, contains more than $\ell$ points of $P$ is that no line is contained in more 
than $\ell$ hyperplanes of $P^*$. As in the case of rich hyperplanes, we also know that 
$I(P,\Sigma) \geq mk$.

In other words, after this transform we reach the same problem involving points and hyperplanes
in $\reals^{d+1}$, and we can apply the preceding analysis verbatim with $d+1$ replacing $d$, and
obtain the assertion in Theorem~\ref{thm:spheres}.

%----------------------------------------------------------------------------------------------------------
\section{Discussion}

The problem studied in this work can be considered as a variant in the study of incidences between
points and hyperplanes. As far as we can tell, the results in the previous works that have studied
such problems (e.g., \cite{ApS,BK}) do not imply our results. 

Several open problems arise. For example, are there variants of our assumptions, in $d\ge 4$
dimensions, that imply the existence of an even lower-dimensional flat that contain many points of~$P$?
This does not hold without any further assumptions, because we can place the points of $P$ in
\emph{general position} in some $(d-2)$-flat $g$, and then there are infinitely many $k$-rich 
hyperplanes, for any $k$ (all hyperplanes that contain $g$), but no $(d-3)$-flat contains more 
than $d-2$ points of $P$.

%------------------------------------------------------------------

However, by a standard projection argument our main theorem can be extended to lower-dimensional flats in the following manner.

\begin{theorem}\label{t:flats}
 Let $d \geq t \geq 3, k \geq t$ be integers, and $t\leq\alpha < 2t-1$.
Let $P$ be a set of $n$ points in~$\reals^d$, for which the number of 
$k$-rich $(t-1)$-flats is at least $c(n^t/k^{\alpha} + n/k)$, for some sufficiently 
large constant $c$ (depending only on $t$). Then there exists a $(t-2)$-flat 
that contains $\Omega\left(k^{(2t-1-\alpha)/(t-1)}\right)$ points of $P$.
\end{theorem}

\begin{proof}
For $t=d$ the theorem reduces to Theorem \ref{thm:main}.
So we assume that $t < d$. If there exists a $(t-2)$-flat containing at least $k$ points of $P$, the theorem trivially holds, so we can assume that each $(t-2)$-flat contains at most $k-1$ points of $P$. As $k \geq t$, it follows that every $k$-rich $(t-1)$-flat contains $t$ affinely independent points from $P$.

Let $V$ be the set of all $(t-1)$-flats in $\reals^d$ spanned by $t$ affinely independent points from $P$. Note that the cardinality of $V$ is finite and that the set of all $k$-rich $(t-1)$-flats is contained in $V$. 
We consider an orthogonal projection $\varphi$ onto a generic $\reals^t$, where generic means that no two points of $P$ are identified and that for any $F \in V$, $\varphi(F)$ is a $(t-1)$-flat.

After the projection, we have a set $\varphi(P)$ of $n$ points in~$\reals^t$ and the number of $k$-rich $(t-1)$-flats is still at least $c(n^t/k^{\alpha} + n/k)$, as a projection of a $k$-rich $(t-1)$-flat remains to be a $k$-rich $(t-1)$-flat. So we apply Theorem \ref{thm:main} to $\varphi(P)$ in $\reals^t$ and obtain a $(t-2)$-flat $G$ in $\reals^t$ that contains $\Omega\left(k^{(2t-1-\alpha)/(t-1)}\right)$ points of $\varphi(P)$. The preimages of these points are distinct, and we claim, moreover, that they lie on a $(t-2)$-flat in~$\reals^d$, which finishes the proof. 
Indeed, denote by $P'$ the intersection of $P$ with  $\varphi^{-1}(G)$, and assume to the contrary that $P'$ does not lie on a $(t-2)$-flat in $\reals^d$. Hence, there are $t$ affinely independent points in~$P'$, which span a $(t-1)$-flat $F'$ in $\reals^d$. It follows that $\varphi(F') \subseteq G$, which is a  contradiction with our generic choice of $\varphi$ as $F' \in V$ is mapped to a $(t-2)$-flat $G$.
\end{proof}

%------------------------------------------------------------------

Other problems, already mentioned earlier, are to obtain upper bound constructions, other than
the one in Section~\ref{sec:up}, for other values of $k$ and of $\alpha$.

%----------------------------------------------------------------------------------------------------------
\section*{Acknowledgements}
The authors thank Peyman Afshani for sharing his thoughts with us concerning this problem.

\end{document}